\documentclass[11pt,a4paper,oneside]{amsart}
\usepackage{color,amssymb,latexsym,amsfonts,textcomp,lastpage,fancyhdr,calc,graphicx}
\usepackage{amsmath,amstext,amsthm,amssymb,amsxtra}
\usepackage{txfonts}
\usepackage[left=2.5cm,right=2.5cm,bottom=2cm,top=2cm]{geometry} 

\usepackage[T1]{fontenc}
\usepackage[pdftex,pdfpagelabels,bookmarks,hyperindex,hyperfigures]{hyperref}

\newtheorem{theorem}{Theorem}

\newtheorem{lemma}{Lemma}

\theoremstyle{definition}

\newcommand{\beql}[1]{\begin{equation}\label{#1}}
\newcommand{\eeq}{\end{equation}}
\newcommand{\comment}[1]{}

\newcommand{\Abs}[1]{{\left|{#1}\right|}}

\newcommand{\Set}[1]{{\left\{{#1}\right\}}}

\newcommand{\RR}{{\mathbb R}}

\newcommand{\CC}{{\mathbb C}}
\newcommand{\ZZ}{{\mathbb Z}}

\newcommand{\QQ}{{\mathbb Q}}

\newcommand{\supp}{{\rm supp\,}}

\newcommand{\ft}[1]{\widehat{#1}}

\newcounter{rem}
\newcounter{step}
\newcounter{mysec}
\newcounter{mysubsec}[mysec]
\newcounter{othm}
\def\theothm{\Alph{othm}}

\begin{document}

\title{Fourier pairs of discrete support with little structure}

\author[M. Kolountzakis]{{Mihail N. Kolountzakis}}
\address{M.K.: Department of Mathematics and Applied Mathematics, University of Crete, Voutes Campus, GR-700 13, Heraklion, Crete, Greece}
\email{kolount@uoc.gr}
\thanks{This work has been partially supported by the ``Aristeia II'' action (Project
FOURIERDIG) of the operational program Education and Lifelong Learning
and is co-funded by the European Social Fund and Greek national resources.}

\begin{abstract}
We give a simple proof of the fact, first proved in a stronger form in \cite{lev2014quasicrystals}, that there exist measures on the real line of discrete support, whose Fourier Transform is also a measure of discrete support, yet this Fourier pair cannot be constructed by repeatedly applying the Poisson Summation Formula finitely many times. More specifically the support of both the measure and its Fourier Tranform are not contained in a finite union of arithmetic progressions.
\end{abstract}

\maketitle



\section{Introduction}

The Poisson Summation Formula (PSF) in dimension 1
\beql{psf}
\sum_{n \in \ZZ} \ft{f}(n) = \sum_{n \in \ZZ} f(n),
\eeq
where $\ft{f}(\xi) = \int_\RR e^{-2\pi i \xi x} f(x)\,dx$ is the Fourier Transform of $f \in {\mathcal S}(\RR)$ (the Schwarz space),
is a fundamental result. It can be rephrased in the language of distributions in the form
\beql{psfd}
\left( \sum_{n \in \ZZ} \delta_n \right)^\wedge = \sum_{n \in \ZZ} \delta_n,
\eeq
where $\delta_a$ denotes the unit point mass at $a$ and the Fourier Transform is meant in
the sense of tempered distributions.

We note that \eqref{psfd} shows us a pair of measures which are (a) the Fourier Transform of one another and (b) they are both of discrete support in $\RR$. In this base case we have that the measure $\mu = \sum_{n \in \ZZ} \delta_n$ is the Fourier Trasnform of itself, but if one applies the standard properties of translation, modulation and dilation with respect to the Fourier Trasnform, one can get, by repeatedly applying the PSF, a multitude of measures of discrete support whose Fourier Trasnform is also a measure of discrete support. A basic question is then whether all such Fourier Transform pairs of discrete measures are due to the PSF. There has been a long series of results over the last 50 years in which it is proved, under varying assumptions, that this is indeed the case  (see the references and discussion in \cite{lev2014quasicrystals}). In the latest such result \cite{lev2013measures,lev2013quasicrystals}, one of the strongest (along perhaps with that in \cite{meyer1970nombres}) and with very weak assumptions, it is proved that if
\beql{ftpair}
\mu = \sum_{s\in S} c_s \delta_s,\ \ \ 
\ft{\mu} = \sum_{t \in T} d_t \delta_t,\ \ \ (c_s, d_t \neq 0)
\eeq
is a Fourier Transform pair of measures, and the sets $S$ and $T$ are both {\em uniformly discrete} (a set is called \textit{uniformly discrete} if the distance between any two of its points is bounded below by a positive constant), then both $T$ and $S$ are periodic sets and, furthermore, this Fourier pair can be derived by a finite number of applications of the PSF.

In the opposite direction it was recently proved \cite{lev2014quasicrystals} that there are Fourier pairs of the form \eqref{ftpair}
with $S, T \subseteq \RR$ being discrete (but not uniformly discrete) sets,
such that $S$ does not contain an infinite part of {\em any} arithmetic progression
and both $S$ and $T$ cannot be covered by finitely many arithmetic progressions.
This strongly disproves that all Fourier pairs of the form \eqref{ftpair} can be constructed by finitely many repeated applications of the PSF \eqref{psfd}.

In this paper we give another construction of a Fourier pair \eqref{ftpair} which cannot arise from finitely many applications of the PSF. Our simple construction does not give sets which intersect any arithmetic progression in a finite set as the construction in \cite{lev2014quasicrystals} does. It compensates for this by being elementary and short.

\begin{theorem}\label{th:main}
There is a translation bounded measure $\nu$ of the form $\nu = \sum_{\lambda\in\Lambda} c_\lambda \delta_\lambda$, ($c_\lambda\neq 0$) such that $\Lambda \subseteq \RR$ is a discrete set and such that $\ft{\nu}$ is also a translation bounded measure of the form $\ft{\nu} = \sum_{s \in S} d_s \delta_s$, ($d_s \neq 0$) where $S$ is also a discrete set and such that both $\Lambda$ and $S$ are not contained in finite unions of arithmetic progressions. Therefore this Fourier pair cannot be derived by finitely many applications of the PSF.
\end{theorem}

\section{The Construction}

They key ingredient is the construction, given $M>0$, of a Fourier pair such as \eqref{ftpair} so that both $\mu$ and $\ft{\mu}$ vanish in the interval $(-M, M)$. The final measures then are costructed by overlaying such measures, for infinitely many values $M \to \infty$. The increasing gaps around the origin serve to keep the end result supported by a discrete set.

Let $\ZZ_N$ denote the cyclic group of $N$ elements.

\begin{lemma}\label{lm:zn}
There is a function $f:\ZZ_N\to\CC$, not identically zero, such that both the function and its Fourier transform $\ft{f}:\ZZ_N\to\CC$ vanish in the interval
$$
I=\Set{x\in\ZZ_N: \Abs{x} \le \frac{N}{10}}.
$$
\end{lemma}
\begin{proof}
We search for $f:\ZZ_N\to\CC$ which is 0 on $I$ such that $\ft{f}$ also vanishes on $I$. This is a homogeneous linear system (the unknowns are the values of $f$ off $I$) with more unknowns ($\sim 4N/5$ of them) than equations ($\sim N/5$ of them) so there is a non-zero solution.
\end{proof}

\begin{lemma}\label{lm:comb}
Suppose $M>1$ is an integer. Then there is a non-zero measure $\mu$ of the form
$$
\mu = \sum_{n \in \ZZ} a_n \delta_{An},
$$
where $A$ is a positive real number, whose Fourier Transform is a measure $\ft{\mu}$ of the form
$$
\ft{\mu} = \sum_{n \in \ZZ} b_n \delta_{Bn},
$$
where $B$ is a positive real number, and such that both $\mu$ and $\ft{\mu}$ are 0 in the interval $(-M, M)$.

Furthermore the measures $\mu$ and $\ft{\mu}$ can be taken to be periodic and the numbers $A$ and $B$ may be chosen to be rational.
\end{lemma}
\begin{proof}
Let us start with the function $f:\ZZ_N\to\CC$ of Lemma 1, where $N=100 M^2$.
Define first the measure
$$
\tau = \sum_{n\in\ZZ} \tau_n \delta_n,
$$
where $\tau_n = f(n \bmod N)$ for all $n\in\ZZ$.
The measure $\tau$ is $N$-periodic, so its Fourier Transform is of the form
$$
\ft{\tau} = \sum_{n \in \ZZ} \ft{\tau}_n \delta_{n/N}.
$$
It is easy to see using the Poisson Summation Formula that $\ft{\tau}_n = \ft{f}(n \bmod N)$, for $n\in\ZZ$ (we follow the normalization $\ft{f}(n) = (1/N)\sum_{k=0}^{N-1} f(k) e^{-2\pi i nk/N}$).

It follows that $\tau$ vanishes in the interval $\left(-\frac{N}{10}, \frac{N}{10}\right)$ and $\ft{\tau}$ vanishes in the interval $\left(-\frac{1}{10}, \frac{1}{10}\right)$.

Define the measure $\mu$ to be the dilate (shrinking) of $\tau$ by $1/\sqrt{N}$
$$
\mu =\sum_{n \in \ZZ} \tau_n \delta_{n/\sqrt{N}}.
$$
It follows that
$$
\ft{\mu} = 
  \sqrt{N} \sum_{n\in\ZZ} \ft{\tau}_n \delta_{n/\sqrt{N}}.
$$
and, therefore, both $\mu$ and $\ft{\mu}$ vanish in the interval $\left(-\frac{\sqrt{N}}{10}, \frac{\sqrt{N}}{10}\right)=(-M,M)$.
\end{proof}

\begin{proof}[Proof of Theorem \ref{th:main}]
Take a sequence $M_n \to \infty$ and apply repeatedly Lemma \ref{lm:comb} to obtain a sequence of periodic measures $\mu_n$ of discrete support, having also $\ft{\mu_n}$ be periodic and of discrete support and such that both $\mu_n$ and $\ft{\mu_n}$ vanish in the interval $(-M_n, M_n)$.

Denote by $T_r$ the translation by $r$ and by $M_a$ the modulation operator by $a$.
Let $\epsilon_n \to 0$ be a $\QQ$-linearly independent sequence.
Each measure $\mu_n$ or $\ft{\mu_n}$ has bounded total variation in any interval of unit length (since they are periodic), say by $V_n$.
Define $D_n = V_n n^2$.

Consider the measure 
$$
\nu = \sum_{n\ge 1} \frac{1}{D_n} M_{\epsilon_n} T_{\epsilon_n} \mu_n,
$$
whose Fourier Transform is the measure
$$
\ft{\nu} = \sum_{n \ge 1} \frac{1}{D_n} T_{\epsilon_n} M_{-\epsilon_n} \ft{\mu_n}.
$$
It follows that $\nu$ and $\ft{\nu}$ have bounded total variation in any interval of unit length.

We now show that the support of both $\nu$ and $\ft{\nu}$ is discrete. Let $J = (a,b)$ be any interval. Then there is an index $n_0$ such that for $n \ge n_0$ we have $(a,b) \subseteq (-M_n+1, M_n-1)$, therefore the support of $\nu$ or $\ft{\nu}$ in $J$ comes only from the contributions of the measures $\mu_1, \mu_2, \ldots, \mu_{n_0}$ or $\ft{\mu_1}, \ft{\mu_2}, \ldots, \ft{\mu_{n_0}}$ and consists therefore of a finite number of points. Hence both $\supp{\nu}$ and $\supp{\ft{\nu}}$ are discrete.

To show that $\supp\nu$ and $\supp\ft\nu$ are not contained in finite unions of arithmetic progressions observe that the subsets of any such union are of finite dimension over $\QQ$, whereas our sets are infinite dimensional over $\QQ$ owing to the $\QQ$-linear independence of the numbers $\epsilon_n$.
\end{proof}

\bibliographystyle{abbrv}
\bibliography{spectral-sets}

\end{document}